\documentclass[11pt]{article}
\usepackage{amsthm}
\usepackage{amsfonts}
\usepackage[T1]{fontenc}
\usepackage{epsfig}

\usepackage[lflt]{floatflt}
\usepackage{epsfig}
\usepackage{graphicx}
\usepackage{amsmath}
\usepackage{amssymb}
\usepackage{color}
\usepackage{subfigure}

\usepackage{makeidx}
\usepackage{latexsym}

\usepackage{anysize}
\marginsize{2cm}{2cm}{1.5cm}{1.5cm}

\newtheorem{note}{Note}
\newtheorem{theorem}{Theorem}
\newtheorem{definition}{Definition}
\newtheorem{proposition}{Proposition}
\newtheorem{lemma}{Lemma}

\newtheorem{remark}{Remark}

\newcommand{\bbN}{{\mathbb N}}

\newcommand{\dd}{{\rm d}}

\newcommand{\bbR}{{\mathbb R}}

\newcommand{\al}{\alpha}
\newcommand{\la}{\lambda}
\newcommand{\p}{\partial}
\newcommand{\be}{\beta}
\newcommand{\G}{\Gamma}
\newcommand{\var}{\varphi}
\newcommand{\DC}{_0^CD^\al_t}

\newcommand{\bx}{\textit{\textbf{x}}}

\begin{document}
\begin{center}\emph{}
\LARGE
\textbf{About the Convergence of a Family of Initial Boundary Value Problems for a Fractional Diffusion Equation with Robin Conditions}
\end{center}

                   \begin{center}
                  { Isolda Cardoso$^1$, Sabrina D. Roscani$^{2,3}$ and Domingo A. Tarzia$^{2,3}$\\
$^1$ Depto. Matem\'atica, ECEN, FCEIA, UNR, Pellegrini 250,  Rosario, Argentina\\
$^2$  Depto. Matem\'atica, FCE, Univ. Austral, Paraguay 1950, S2000FZF Rosario, Argentina \\
$^3$ CONICET, Argentina\\
(isolda@fceia.unr.edu.ar, sroscani@austral.edu.ar, dtarzia@austral,edu.ar)}
                   \vspace{0.2cm}

       \end{center}
      
\small

\noindent \textbf{Abstract:} We consider a family of initial boundary value problems governed by a fractional diffusion equation with Caputo derivative in time, where the parameter is the Newton heat transfer coefficient linked to the Robin condition on the boundary. For each problem we prove existence and uniqueness of solution by a Fourier approach. This will enable us to also prove the convergence of the family of solutions to the solution of the limit problem, which is obtained by replacing the Robin boundary condition with a Dirichlet boundary condition.\\



\section{Introduction}\label{Sec:Introduction}

\par Initial Boundary Value Problems (IBVPs) for the Fractional Diffusion Equation (FDE) have been profusely studied in the last years \cite{KuYa:2018,Lu:2010,Povs:2014,Pskhu:2003,SaYa:2011,Za:2013}. It is well known that the FDEs describe subdiffusion processes. That is, diffusion phenomena where the mean squared displacement of the particle is proportional to $t^{\al/2}$ instead of being proportional to $t^{1/2}$, as occurs when diffusion processes take place (see \cite{KlSo:2005, MeKl:2000} and references therein).


\par Let $\Omega$ be a bounded domain in $\bbR^d$ with sufficiently smooth boundary $\p\Omega$ and let $T>0$. Let $\Delta$ denote the usual Laplacian operator on $\bbR^d$. In this work we will study the IBVP associated to the FDE with the Caputo Derivative given by
\begin{equation}\label{FDE}
\DC u(\bx,t)= \Delta u(\bx,t), \qquad (\bx,t) \in \Omega \times (0,T),
\end{equation}
where $\al \in (0,1)$ and $\DC$ denotes the Caputo fractional derivative of order $\alpha$ in the $t$ variable. This is defined for every absolutely continuous function $ f \in  AC[0,T]$ by 
\begin{equation}\label{Def-Der-Cap}
\,^C_{0} D^{\alpha}_t \,f(t)= \frac{1}{\Gamma(1-\al)}\displaystyle\int^{t}_{0}(t-\tau)^{-\al} f'(\tau)\dd\tau. \end{equation}
\noindent We will adress the problem with a convective condition,  that is, a boundary condition where the incoming flux is proportional to the temperature difference between the surface of the material and an imposed ambient temperature. They are also called Robin conditions, and they involve a linear combination of Dirichlet and Neumann conditions:
\begin{equation}\label{Robin-cond-presentation}
\, -\frac{\p u}{\p n}(\bx,t)=\beta \left[u(\bx,t)-u_\infty \right],\qquad     \bx\in \p \Omega, \,  0<t<T, \end{equation}
where   $\frac{\p}{\p n}$ denotes the unitary exterior normal derivative on $\p \Omega$, $u_\infty$ is the external imposed temperature and the parameter $\be$ is the Newton coefficient of heat transfer, which is a positive constant that will play an important role in later analysis. 

\par More precisely, the IBVP for FDE in time and Robin condition with parameter $\be \in \bbR^+$ and initial data $u_0 \in \mathcal{C}(\Omega)$ that we will consider is the following 
\begin{equation}{\label{beta-IBVP}}
\begin{array}{clll}
     (i)  &  \DC u(\bx,t)= \Delta u(\bx,t), &   \bx \in \Omega, \,  0<t<T,   \,\\
     (ii)  &  u(\bx,0)=u_0(\bx), &   \bx \in \Omega \, \,\\
     (iii) &  \frac{\p u}{\p n}(\bx,t)+\beta u(\bx,t)=0 &   \bx \in \p \Omega, \,  0<t<T.                      \end{array}
                                             \end{equation}
Notice that all the physical parameters involved in heat transfer, such as density, thermal conductivity or specific heat are considered constant equal one for the sake of simplicity, and the imposed temperature $u_\infty=0$. We will refer to problem \eqref{beta-IBVP} as the $\beta-$IBVP.

\par Second-order parabolic equations in multidimensional domains, are usually treated by using variational calculus techniques (see for example 
\cite{Brezis,Evans,Ladyzehnskaya}  among several others), and one of the key tools used in the proofs (for example, proofs for existence and uniqueness) is the application of the following property
\begin{equation}\label{regla-cadena} \frac{d}{dt}|f(t)|^2=2\left(\frac{d}{dt}f(t),f(t) \right) \end{equation}
where the notation on the right hand side denotes an inner product and the left hand side has the derivative of the squared of the norm,  in an appopiated Hilbert space. The validity of property \eqref{regla-cadena} follows from the Chain Rule. However, when working with fractional derivatives in time, we cannot deduce the same rule. Indeed, an analogous version of expression \eqref{regla-cadena} for Caputo derivatives is the given in \cite[Theorem 2.1]{Za:2009} but in contrast to the simplicity of  \eqref{regla-cadena}, the left hand side contains, besides the fractional derivative term, other terms involving integrals depending on time with singular kernels. 
We will, however, use some variational techniques throughout the process related to the spacial variable. Let us also recall that in \cite{Ke:2011} a unique existence result of strong solution as the linear combination of the single-layer potential, the volume potential, and the Poisson integral, is given for the same problem with Robin conditions. We will also not choose this path.

\par In our work, we will obtain existence and uniqueness of solution for each $\be-$IBVP by a Fourier approach. Luchko in \cite{Lu:2010} and Sakamoto and Yamamoto in \cite{SaYa:2011} took this approach for a more general operator than the Laplacian. They established the existence of a unique weak solution of an IBVP for the FDE \eqref{FDE} with Dirichlet boundary conditions in a bounded domain in $\bbR^d$. We will also be considering this Dirichlet problem, so we state it precisely here:
\begin{equation}\label{D-IBVP}
\begin{array}{clll}
     (i)  &  \DC u(\bx,t)= \Delta u(\bx,t), &   \bx \in \Omega, \,  0<t<T,   \,\\
     (ii)  &  u(\bx,0)=u_0(\bx), &   \bx \in \Omega \, \,\\
     (iii) &   u(\bx,t)=0 &   \bx \in \p \Omega, \,  0<t<T;                      \end{array}
                                             \end{equation}
where $\Omega$ is a bounded domain in $\bbR^d$ with Lipschitz continuous boundary $\p\Omega$. We will refer this problem as the $D-$IBVP.
 
\par As a consequence of this approach, since the proofs are based on the eigenfunction expansions, we will be able to study the asymptotic behavior of the solutions of the $\be-$IBVP when $\be$ goes to $\infty$ and compare them to the solution of the $D-$IBVP. In doing this, we will make use of the variational techniques we slighted before, and recall the works of Filinovsky in \cite{Fi:2014} and \cite{Fi:2017}. The main theorems we will prove are the following:
\begin{theorem}\label{Main1}
Let $\be>0$. If $u_0 \in  L^2(\Omega)$ then there exists a unique weak solution $u_{\be}\in C([0,T];L^2(\Omega)) \cap C((0,T];H^1(\Omega))$ to the $\be$-IBVP \eqref{beta-IBVP} such that $\DC u_{\be} \in C((0,T];L^2(\Omega))$. Moreover,  
\begin{align}\label{cota-u-beta-L2} ||u_{\be}||_{C([0,T];L^2(\Omega))} \le & C ||u_{0}||_{L^2(\Omega)}, \end{align} and we have \begin{equation}\label{series-u-beta} u_{\be}(\bx,t) =  \sum\limits_{n=1}^{\infty} (u_0,\psi_{n}(\beta;\cdot))E_{\al,1}(-\lambda_{n}(\beta)t^{\al})\psi_{n}(\beta;\bx) \end{equation} in $C([0,T];L^2(\Omega)) \cap C((0,T];H^2(\Omega)\cap H^1(\Omega))$.
\end{theorem}
\begin{theorem}\label{Main2}
The family of solutions $\left\{u_\be\right\}$ of the $\be-IBVP$ \eqref{beta-IBVP}, converges in $L^2(\Omega)$ strong to the solution $u_D$ to the $D-IBVP$ when $\be \rightarrow \infty$, for every $t\in (0,T)$.   
\end{theorem}

The rest of this manuscript is structured as follows. In Section 2 some basics on fractional calculus  and Sturn-Liouville theory  related to the temporal and spacial variables respectively are given, as well as the definition of weak solution related to the Fourier approach. In Section 3 we study the $\beta-$IBVP by using the Fourier approach and Theorem \ref{Main1} is proved. Finally, the convergence of the eigenvalues and the eigenfunctions of the $u_\be$ solutions are given in Section 4, leading to convergence in $L^2(\Omega)$, for every $t\in (0,T)$, of the $u_\be$ solutions  to the $u_D$ solution when  $\be \to \infty$ as stated in Theorem \ref{Main2}. Finally, in order to illustrate the convergence result given in Section 4, we aboard the one-dimensional case and give some examples using SageMath software in Section 5.
 
\section{Preliminaries}\label{Preliminaries}
\subsection{Temporal variable: Fractional Calculus and functions involved}

\begin{definition}\label{defiIntFrac} Let $0<\al< 1$ be.    For $f \in L^1(0,T)$, we define the \textsl{fractional 
Riemann--Liouville integral of order  $\alpha$} as
$$_{0}I^{\alpha}f(t)=\frac{1}{\Gamma(\alpha)}\int^{t}_{0}(t-\tau)^{\alpha-1} f(\tau)\dd\tau. $$
\end{definition}
\begin{note} The Caputo  derivative of order $\alpha$ defined in \eqref{Def-Der-Cap} can be expressed in terms of the fractional integral of Riemann-Liouville  of order $1-\al$ by 
$$\,^C_{a} D^{\alpha}f(t)= \left[ \, _{0}I^{1-\alpha}(f')  \right] (t)=\frac{1}{\Gamma(1-\alpha)}\int^{t}_{0}(t-\tau)^{-\alpha} f'(\tau)\dd\tau, $$
 for every $t\in (0,T).$
\end{note}

\par When dealing with fractional derivatives, it is widely known the importance of the Mittag-Leffler functions and its properties. We recall them now.

\begin{definition} For every $t \in \bbR^+_0$ and $\al>0$, the \textit{Mittag-Leffler} function is defined by
\begin{equation}\label{E} E_{\al}(t)=\displaystyle\sum_{k=0}^{\infty}\frac{t^k}{\G(\a k+1)}.\end{equation}
\end{definition}


\begin{proposition}\label{PropsE}  For  $\al \in (0,1)$ and $\lambda\in\bbR$ the next assertions follow:
\begin{enumerate}
\item If $\lambda>0$, then $E_{\al}(-\lambda t^\al)$ is a positive decreasing function for $t\in\bbR^+_0$ such  that 
\begin{equation}\label{ML<1}
E_{\al}(-\la t^\al)\leq 1.  
\end{equation} 
Moreover, there exists a constant $C>0$ such that for every $t\in\bbR^+_0$,
\begin{equation}\label{acotaML}
E_{\al}(-\la t^\al)\leq \frac{C}{1+\la t^\al}. 
\end{equation}
\item   For every $t\in \bbR^+_0$ and $\lambda>0$,
\begin{equation}\label{ML-prop-DC} _0^CD^\al E_{\al}(-\lambda t^\al)=-\lambda E_{\al}(-\lambda t^\al).
\end{equation}  
\end{enumerate}
\end{proposition}
\begin{proof} The first part easily follows from the definition. See \cite[Corollary 3.7]{GKMR:2014} for the estimate \eqref{acotaML} and \cite[Theorem 4.3]{Diethelm} for item 2. 
\end{proof}

\subsection{Space variable: Sobolev spaces, variational formulation}

\par Let us denote the usual $L^2(\Omega)$ inner product and the usual associated norm by $(u,v)$ and $||u||_{L^2(\Omega)}$, respectively. For the standard Sobolev spaces $H^{m}(\Omega)$ we consider the inner product \begin{align*}(u,v)_{H^{m}(\Omega)} = & \sum\limits_{0\le|\sigma|\le m} (D^\sigma u,D^\sigma v) \end{align*} and the respective associated norm \begin{align*} ||u||_{H^{m}(\Omega)} = &\left( \sum\limits_{0\le |\sigma| \le m} ||D^{\sigma}u||^2_{L^{2}(\Omega)}\right)^{1/2}. \end{align*} Also, the space $H_{0}^{1}(\Omega)$ denotes the closure of $C^{\infty}_{c}(\Omega)$ in $H^{1}(\Omega)$. 
\par We will work with some other bilinear forms, inner products and norms, let us describe them below and give some references.
%

\par In $H^1_0(\Omega)$ we have the bilinear form defined by
\begin{equation}\label{a(u,v)}
\begin{array}{rcl}
a\colon & H^1_0(\Omega)\times H^1_0(\Omega) & \rightarrow \bbR\\
 & (u,v) & \rightarrow a(u,v)= \int\limits_{\Omega} \nabla u\nabla v \dd x
\end{array} 
  \end{equation}
and in $H^1(\Omega)$ we have the bilinear form defined by
\begin{equation}\label{a_beta(u,v)}
\begin{array}{rcl}
a_{\be}\colon & H^1(\Omega)\times H^1(\Omega) & \rightarrow \bbR\\
 & (u,v) & \rightarrow a_\beta(u,v)= \int\limits_{\Omega} \nabla u\nabla v \dd x +  \beta \int\limits_{\partial\Omega} uv \dd\gamma.
\end{array}
\end{equation}

\par Note that in $H^1(\Omega)$ we have that $||v||_{H^{1}(\Omega)} = [a(v,v)+(v,v)]^{1/2}$.

\par The next Lemma was proved in \cite{TaMathNotae:1979,Ta:1979} and it is a very useful tool. 
\begin{lemma}\label{equiv-norm-H1} There exists a constant $\eta_1>0$ such that 
\begin{equation}\label{LemaTop}
a(v,v)+ \int\limits_{\partial\Omega} v^2 \dd\gamma \geq \eta_1 ||v||^2_{H^1(\Omega)} \qquad \forall \, v\, \in H^1(\Omega). 
\end{equation}

Moreover, the induced norm $\sqrt{a_\be(\cdot,\cdot)}$ in $H^1(\Omega)$ is equivalent to the  classical norm in $H^1(\Omega)$ and there exists a constant $\eta_\be= \eta_1\cdot \min\left\{ 1,\be\right\}$ such that 
\begin{equation}\label{lemaTop2}
a_\be(v,v) \geq \eta_\be ||v||^2_{H^{1}(\Omega)} \qquad \forall \, v\, \in H^1(\Omega). 
\end{equation}

\end{lemma}

The next theorem is well known (see \cite[Theorem 9.31]{Brezis}), and we present it for the benefit of the reader.
\begin{theorem}\label{ExBaseH-PbST-D}
There is a Hilbert basis $\{\varphi_{n}(\cdot)\}_{n\ge 1}$  of $L^{2}(\Omega)$ and a sequence of real positive numbers $\{\lambda^D_{n}\}_{n\ge 1}$, $0<\lambda^D_{1}<\lambda^D_{2}\leq \ldots$ enumerated
according to their multiplicities,  with $\lambda^D_{n}\to \infty$ such that 
\begin{itemize}
\item[(i)]\label{i-D} $\varphi_{n}(\cdot)\in H_0^{1}(\Omega)\cap C^{\infty}(\Omega)$ and 
\item[(ii)]\label{ii-D} $\Delta \varphi_{n}(\bx)+\lambda^D_{n}\varphi_{n}(\bx)=0$, $\bx \in $  $\Omega$.
\end{itemize}
The $\lambda^D_{n}$'s are the eigenvalues of $-\Delta$ with Dirichlet boundary condition, and the $\varphi_n$'s are the associated eigenfunctions.
\end{theorem}

\par Following the same variational techniques and compact operators arguments, it is straightforward to derive the next result which is analogous to Theorem \ref{ExBaseH-PbST-D} for the $\beta-IBVP$:
\begin{theorem}\label{ExBaseH-PbST-beta}
For every $\beta>0$, there is a Hilbert basis $\{\psi_{n}(\be,\cdot)\}_{n\ge 1}$  of $L^{2}(\Omega)$ and a sequence of real positive numbers $\{\lambda_{n}(\be)\}_{n\ge 1}$, $0<\lambda_{1}(\be)<\lambda_{2}(\be)\leq \ldots$ enumerated
according to their multiplicities, with $\lambda_{n}(\be)\to \infty$ such that 
\begin{itemize}
\item[(i)]\label{i-beta} $\psi_{n}(\be,\cdot)\in H^{1}(\Omega)\cap C^{\infty}(\Omega)$,
\item[(ii)]\label{ii-beta} $\Delta \psi_{n}(\be,\bx)+\lambda_{n}(\be)\psi_{n}(\be,\bx)=0$, $\bx \in $  $\Omega$, and
\item[(iii)]\label{iii-beta} $\frac{\partial}{\partial n} \psi_{n}(\be,\bx) + \beta \psi_{n}(\be,\bx)=0$ for every  $ \bx\in \partial\Omega$.
\end{itemize} 
The $\lambda_{n}(\be)$'s are the eigenvalues of $-\Delta$ with Robin boundary condition of parameter $\be$, and the $\psi_n(\be,\cdot)$'s are the associated eigenfunctions.
\end{theorem}

\subsection{Weak solutions.}
Following the work of \cite{SaYa:2011}, we define now the idea of weak solution related to problems \eqref{D-IBVP}, \eqref{beta-IBVP} and the Fourier approach proposed. These are somewhat technical definitions but we need them for the sake of rigurosity. Let us consider the space
\begin{equation}\label{W1t}
W^1_t(\Omega\times [0,T])=\left\{w\,\colon\, w(x,\cdot) \in AC[0,T]\right\}.
\end{equation}
For the Fractional Diffusion Equation we have:
\begin{definition}\label{weakSolFDE}  We say that a function $w:[0,T]\rightarrow L^2(\Omega)  $ is a \textit{weak solution} to the FDE \eqref{FDE} if there exists a sequence $(w_n)_{n\in \bbN}$ of smooth functions in $C(\bar{\Omega}\times [0,T])\cap C^{2}_x(\Omega\times (0,T))\cap W^1_t(\Omega\times [0,T])$ such that $w_n$ verifies the FDE \eqref{FDE} in the classical sense for every $n\in \bbN$, and: 
\begin{itemize}
\item[(i)] $w_n \to w$ in $L^2(\Omega)$, a.e. in (0,T).
\item[(ii)] There exists $v \in C\left( (0,T); L^2(\Omega)\right)$ such that  $\lim\limits_{n\to \infty} \left\|\DC w_n (\cdot,t)-v(\cdot.t) \right\|_{L^2(\Omega)}=0 $ a.e. in $ (0,T)$.  In such a case we will write $v=\, \DC w$. 
 \item[(iii)] $\Delta w \in C((0,T);L^2(\Omega))$ and $ \DC w=\Delta w \quad \text{in }\, L^2(\Omega), \, a.e. \, (0,T).$ 
\end{itemize}
\end{definition}

Next, for the Dirichlet Initial Boundary Value Problem we have:

\begin{definition}\label{weakSolD-IBVP} We say that $u$ is a \textit{weak solution} to problem \eqref{D-IBVP} if there exists a sequence $(u_n)_{n\in \bbN}$ of functions in $C(\bar{\Omega}\times [0,T])\cap C^{2}_x(\Omega\times (0,T))\cap W^1_t(\Omega\times [0,T])$ such that 
\begin{itemize}
\item[(i)] There exist a sequence $(u_{n0})_n$ in $C(\bar{\Omega})$ such that $ u_{n0}\to u_0 \, \in \, L^2(\Omega)$ and for every $n\in \bbN$, $u_n$ is a classical solution to the  approximate problem
\begin{equation}\label{D-IBVP-aprox}
\begin{array}{clll}
     (i)  &  \DC u_n(\bx,t)= \Delta u_n(\bx,t), &   \bx \in \Omega, \,  0<t<T,   \,\\
     (ii)  &  u_n(\bx,0)=u_{n0}(\bx), &   \bx \in \Omega \, \,\\
     (iii) &   u_n(\bx,t)=0 &   \bx \in \p \Omega, \,  0<t<T.                      \end{array}
                                             \end{equation}
\item[(ii)] $\DC u$ exists in the sense defined in Definition \ref{weakSolFDE} and $u$ is a weak solution to \eqref{FDE}.
\item[(iii)] $u\in C([0,T];L^2(\Omega))$ and  
$ \lim\limits_{t\rightarrow 0}\left\| u(\cdot,t)-u_0 \right\|_{L^2(\Omega)}=0$.
\end{itemize}
\end{definition}     

And finally, for the Robin Initial Boundary Value Problem we have:

\begin{definition}\label{weakSol-beta-IBVP} We say that $u$ is a \textit{weak solution} to problem \eqref{beta-IBVP} if there exists a sequence $(u_n)_{n\in \bbN}$ of functions in $C(\bar{\Omega}\times [0,T])\cap C^{2}_x(\Omega\times (0,T))\cap W^1_t(\Omega\times [0,T])$ such that 
\begin{itemize}
\item[(i)] There exist a sequence $(u_{n0})_n$ in $C(\bar{\Omega})$  such that $ u_{n0}\to u_0 \, \in \, L^2(\Omega)$ and for every $n\in \bbN$, $u_n$ is a classical solution to the approximate problem
\begin{equation}\label{beta-IBVP-aprox}
\begin{array}{clll}
     (i)  &  \DC u_n(\bx,t)= \Delta u_n(\bx,t), &   \bx \in \Omega, \,  0<t<T,   \,\\
     (ii)  &  u_n(\bx,0)=u_{n0}(\bx), &   \bx \in \Omega \, \,\\
     (iii) &  \frac{\p}{\p \nu} u_n(\bx,t)+\beta u_n(\bx,t)=0 &   \bx \in \p \Omega, \,  0<t<T.   \end{array}
                                             \end{equation}

\item[(ii)] $\DC u$ exists in the sense defined in Definition \ref{weakSolFDE} and $u$ is a weak solution to  \eqref{FDE}.
\item[(iii)] $u\in C([0,T];L^2(\Omega)) $ and  
$ \lim\limits_{t\rightarrow 0}\left\| u(\cdot,t)-u_0 \right\|_{L^2(\Omega)}=0.$
\end{itemize}
\end{definition}     

\section{The Fourier Approach}

\par Following the lines of the works  \cite{Lu:2010} and \cite{SaYa:2011}, we look for a  solution to problem \eqref{beta-IBVP} constructed analytically by using the Fourier method of variable separation.  Thus, let us recall first Theorem 2.1 of \cite{SaYa:2011}.

\begin{theorem}\label{sol-D-IBVP} If $u_0 \in  L^2(\Omega)$ then there exists a unique weak solution $u_{D}\in C([0,T];L^2(\Omega)) \cap C((0,T];H^2(\Omega)\cap H^1_0(\Omega))$ to the D-IBVP \eqref{D-IBVP} such that $\DC u_D \in C((0,T];L^2(\Omega))$. Moreover, there exists a constant $C>0$ such that 
\begin{equation}\label{cota-uD-L2} ||u_D||_{C([0,T];L^2(\Omega))} \le  C ||u_{0}||_{L^2(\Omega)}, \end{equation} and we have that the solution is given by the series
\begin{equation}\label{series-u-D} u_D(\bx,t) =  \sum\limits_{n=1}^{\infty} (u_0,\varphi_{n})E_{\al}(-\lambda^D_{n}t^{\al})\varphi_{n}(\bx) \end{equation} 

\noindent where $(\la_n^D,\varphi_n)$ are the eigenvalues and eigenfunctions given in Theorem \ref{ExBaseH-PbST-D}.

\end{theorem}

\par As we said above,  we look for a particular solution $u$ of the equation $\eqref{beta-IBVP}-i$ of the form
\begin{equation}\label{sol-part-var-sep}
 u(\bx,t)=\psi (\bx) \eta (t),
\end{equation}
then, by replacing \eqref{sol-part-var-sep} in \eqref{beta-IBVP}$-i$, we are lead to the following equations
\begin{equation}\label{ecs-var-sep}
\frac{_0^C D^\al_t \eta (t)}{\eta(t)}=\frac{\Delta \psi (\bx)}{\psi(\bx)}=-\lambda,
\end{equation}
where $\lambda$ is a constant which does not depend on $\bx$ nor $t$. Thus we are left with two different problems. 

\par The first one is regarding the spatial variable by considering the spatial  equation in  \eqref{ecs-var-sep} together  with the Robin boundary condition, which derives in a classical Sturm-Liouville problem.  And the second one is for the temporal variable which is an ordinary fractional differential equation. Both problems will be acoppled later by using the initial condition \eqref{beta-IBVP} (ii). More precisely, we have problems:
\begin{equation}\label{PD-beta}
\begin{array}{rll}
 (i) & \Delta \psi(\bx)+\lambda \psi(\bx)=0, & \bx \in \Omega, \\
  (ii) & \frac{\p \psi}{\p n}(\bx)+\beta \psi(\bx)=0 &   \bx\in \p \Omega;  
\end{array}
\end{equation}
and 
\begin{equation}\label{ODE-Caputo}
\begin{array}{rll}
 (i) & \DC \eta(t)=-\lambda \eta(t), & t \in (0,T), \\
  (ii) & \eta(0)= 1. &    
\end{array}
\end{equation}

According to Theorem \ref{ExBaseH-PbST-beta} and  Proposition 
\ref{PropsE} item $3$,  it is natural to consider function $u_\be: [0,T]\rightarrow L^2(\Omega)$ defined as a series

\begin{equation}\label{series-u-beta} u_\be(\bx,t) =  \sum\limits_{n=1}^{\infty} (u_0,\psi_{n}(\be;\cdot))E_\al(-\lambda_{n}(\be)t^{\al})\psi_{n}(\be;\bx) \end{equation} 
as a desired solution to problem \eqref{beta-IBVP}, where $\left(\psi_{n}(\be),\lambda_n(\be)\right)$ are the eigenfunctions and eigenvalues given in Theorem \ref{ExBaseH-PbST-beta}.

It is easy to prove that $u_\be$ is well defined. Let us see that $u_\be$ is  a weak solution to equation \eqref{FDE} in the sense of  Definition \ref{weakSolFDE}. For that purpose let 
$w:(0,T]\rightarrow L^2(\Omega)$ be the function defined by
\begin{equation}\label{v} w(\bx,t) =  \sum\limits_{n=1}^{\infty} (u_0,\psi_{n}(\be;\cdot))E_\al(-\lambda_{n}(\be)t^{\al})(-\lambda_n(\be))\psi_{n}(\be;\bx) \end{equation} 
which is well defined. Indeed,  by applying inequality \eqref{acotaML} we get  
\begin{align} \nonumber
||w(\cdot,t)||^2_{L^2(\Omega)} = & \sum\limits_{n=1}^{\infty} |(u_{0}, \psi_{n}(\be;\cdot))|^2 |E_{\alpha}(-\lambda_{n}(\be)t^{\al})|^2 \lambda^2_{n}(\be) \\ \label{cotaDeltau}\le  & \, C\sum\limits_{n=1}^{\infty} |(u_{0}, \psi_{n}(\beta;\cdot))|^2 \left(\frac{\lambda_{n}(\be)}{1+\lambda_{n}(\be)t^{\al}}\right)^2  \leq C ||u_{0}||^2_{L^2(\Omega)} t^{-2\al}.
 \end{align}  
 By the other side, let the sequence $(w_n)_{n\in\bbN}$ be, where for every $n\in \bbN$ the function $w_n:(\Omega \times [0,T])\rightarrow \bbR$ is given by the  finite sum
$$  w_n (x,t)=  \sum\limits_{k=1}^{n} (u_0,\psi_{k}(\be))E_\al(-\lambda_{k}(\be)t^{\al})(-\lambda_k(\be))\psi_{k}(\be; \bx).$$

Clearly, if we define the functions $u_{\be n}=\sum\limits_{k=1}^{n} (u_0,\psi_{k}(\be))E_\al(-\lambda_{k}(\be)t^{\al})\psi_{k}(\be; \bx)$ for every natural $n$,  we have that $\DC u_{\be n}(x,t)= w_n.$ \\
Now, for every $t>0$ we have
\begin{align*}
||w(\cdot,t)- \DC u_{\be n}(\cdot,t)||^2_{L^2(\Omega)} =&  \sum\limits_{k=n+1}^{\infty} |(u_{0}, \psi_{k}(\be; \cdot))|^2 |E_\al(-\lambda_{k}(\be)t^{\al})|^2 \lambda_{k}^2(\be) \\
\leq &   C  t^{-2\al} \sum\limits_{k=n+1}^{\infty} |(u_{0}, \varphi_{k}(\cdot))|^2\rightarrow 0, \text{ if }\, n\to\infty
 \end{align*}
from where we conclude that $w=\,\DC u_\be$ in the sense of Definition \ref{weakSolFDE}.\\
Analogously, we deduce that $\Delta u= w$ and we conclude that 
$$ \DC u_\be =\Delta u_\be \quad \text{in }\, L^2(\Omega), \, a.e. \, (0,T), \quad \text{for every } \be>0,$$ 
that is, $u_\be$ is a weak solution to the FDE \eqref{FDE}.  
\begin{remark} It is straightforward that \eqref{series-u-beta} is a weak solution to problem \eqref{beta-IBVP} according to Definition \ref{weakSol-beta-IBVP}. Note that the boundary conditions are verified from Theorem \ref{ExBaseH-PbST-beta}. \end{remark}

From the previous reasoning we state the following Lemma. 
\begin{lemma}\label{beta-weak-sol} The function \eqref{series-u-beta} belongs to $C([0,T];L^2(\Omega))$ and it is a $weak$ solution to the FDE $\eqref{FDE}$ in the sense of Definition \ref{weakSolFDE}. 
\end{lemma}

\par The proof of the next theorem is obtained by mimicking the steps of the proof of Theorem \ref{sol-D-IBVP} given in \cite{SaYa:2011} and we are lead to a similar result for the $\be-$IBVP. The difference is on the Robin conditions and the spaces of functions involved. However, we include it for the sake of completeness. Thus, with the notation from Theorem \ref{ExBaseH-PbST-beta} we have the following.

 \begin{theorem}\label{sol-beta-IBVP} Let $\be>0$. If $u_0 \in  L^2(\Omega)$ then there exists a unique weak solution $u_{\be}\in C([0,T];L^2(\Omega)) \cap C((0,T]; H^1(\Omega))$ to the $\be$-IBVP \eqref{beta-IBVP} such that $\DC u_{\be} \in C((0,T];L^2(\Omega))$. Moreover, $u_\be$ is given by the series \eqref{series-u-beta} and  
\begin{align}\label{cota-u-beta-L2} ||u_{\be}||_{C([0,T];L^2(\Omega))} \le &  ||u_{0}||_{L^2(\Omega)}, \end{align} 
\end{theorem}

\proof We first prove that, if $u$ is a formal solution to  \eqref{beta-IBVP}(i), (ii) and (iii), then $u$ is the weak solution given by (27).

\par Let $t\in[0,T]$. If $u(\cdot,t)\in L^2(\Omega)$, then from Theorem \eqref{ExBaseH-PbST-beta} we have the following expansion. For $\bx \in \Omega$,
\begin{align}\label{u-formal-en-serie} u(\bx,t) = & \sum\limits_{n=1}^{\infty} (u(\cdot,t),\psi_{n}(\beta;\cdot))\psi_{n}(\beta;\bx) = \sum\limits_{n=1}^{\infty} u_{n}(\be;t)\psi_{n}(\beta;\bx), \end{align}
where we defined 
\begin{align*}
u_{n}(\be;t) = & (u(\cdot,t),\psi_{n}(\beta;\cdot)).
\end{align*}

\par Since $\DC u(x,t) = \Delta u(x,t)$ for $\bx\in\Omega$, it follows that also we have for each $n\in\mathbb{N}$,
\begin{equation*} (\DC u(\cdot,t),\psi_{n}(\beta;\cdot)) = (\Delta u(\cdot,t),\psi_{n}(\beta;\cdot)). \end{equation*}
To extract some information on $u$ from the above equation let us compute both sides and then compare them.
\par For the left side we recall the definition for the Caputo derivative, namely formula \eqref{Def-Der-Cap}. By applying Fubini's Theorem and the derivation under the integral sign Theorem we end up with
\begin{align*}
(\DC u(\cdot,t),\psi_{n}(\beta;\cdot)) = & \DC u_{n}(\be;t).
\end{align*}
For the right side we apply Green's formula and recall the boundary conditions (\ref{beta-IBVP}-iii) for $u$ and conditions (ii) and (iii) in Theorem \ref{ExBaseH-PbST-beta} for $\psi_{n}$ (eigenvalue and boundary condition, respectively), to obtain
\begin{align*}
(\Delta u(\cdot,t),\psi_{n}(\beta;\cdot)) = & - \lambda_{n}(\beta) u_{n}(\be;t).
\end{align*} 
 
\par Thus for each $n\in\mathbb{N}$ we are left with an ordinary fractional differential equation together with the initial condition (\ref{beta-IBVP}-ii), that is to say, the following initial value problem.
\begin{align}\label{FODEIVP}
\begin{array}{rcl}
\DC u_{n}(\be;t) &=& - \lambda_{n}(\beta) u_{n}(\be;t), \qquad \mbox{ for } t>0 \\
u_{n}(\be;0) &=& (u_{0}, \psi_{n}(\beta;\cdot)).
\end{array}
\end{align}
The well known theory (see for example \cite{Podlubny}, or \cite{Kilbas}, among others) provides us with a unique solution for this problem by means of the Mittag-Leffler functions that we defined in Section \ref{Preliminaries}. More precisely,
\begin{align*}
u_{n}(\be;t) = & (u_{0}, \psi_{n}(\beta;\cdot)) E_\al(-\lambda_{n}(\beta)t^{\al}).
\end{align*}
Replacing the above solution in \eqref{u-formal-en-serie} gives us the desired expression for $u$ which coincide with the function given in \eqref{series-u-beta} and then  the formal solution constructed is a weak solution to problem \eqref{beta-IBVP} according to Lemma \ref{beta-weak-sol}.\\

\par Now, from 
\begin{align*}
||u_\be(\cdot,t)||^{2}_{L^2(\Omega)} \le &  \sum\limits_{n=1}^{\infty} |(u_{0}, \psi_{n}(\beta;\cdot))|^2 |E_\al(-\lambda_{n}(\beta)t^{\al})|^2 \le   ||u_{0}||^2_{L^2(\Omega)}.
\end{align*}
we deduce that $u_\be \in L^2(\Omega) $ for every $t \in [0,T]$, moreover we have that $ u_\beta(\cdot,\cdot)\in C([0,T];L^{2}(\Omega))$ and 
$||u_\be||_{C([0,T];L^{2}(\Omega))} \le  ||u_{0}||_{L^2(\Omega)}.$
\par Let us show next that $u_\be \in C((0,T]; H^1(\Omega))$. In order to do this, let us note that the family $\{\psi_{n}(\beta;\cdot)\}_{n\in\mathbb{N}}$ is orthogonal with respect to the norm $\sqrt{a_\beta(\cdot,\cdot)}$, which is equivalent to the $||\cdot||_{H^1}  $ norm as was stated in Lemma \ref{equiv-norm-H1}. Indeed, for every $k,l\in\mathbb{N}$
\begin{align*}
a_{\beta}( \psi_{n}(\beta;\cdot),\psi_{l}(\beta;\cdot) ) & = 
\int\limits_{\Omega} \nabla \psi_{n}(\beta)\nabla \psi_{l}(\beta) \dd x +  \beta \int\limits_{\partial\Omega} \psi_{n}(\beta) \psi_{l}(\beta)\dd\gamma  \\
&= -\int\limits_{\Omega} \Delta \psi_{n}(\beta)\nabla \psi_{l}(\beta) \dd x +\int\limits_{\partial\Omega} \frac{\p \psi_{n}(\beta)}{\p n} \psi_{l}(\beta)\dd\gamma  +  \beta \int\limits_{\partial\Omega} \psi_{n}(\beta) \psi_{l}(\beta)\dd\gamma \\
 & = \int\limits_{\Omega} \lambda_n(\beta) \psi_{n}(\beta)\psi_{l}(\beta) \dd x = \lambda_{n}(\beta) \delta_{nl} 
\end{align*}
where  $\delta_{kl}$ denotes the Kronecker's delta and we have applied Green's theorem, then the fact that every $\psi_k$ verifies (ii) and (iii) in Theorem \ref{ExBaseH-PbST-beta}  and finally that $\{\psi(\beta)_k\}$ is an orthonormal basis in $L^2(\Omega)$.

\par Hence, for every $t>0$ we have
\begin{align*}
a_{\beta}( u_\beta(\cdot,t),u_\beta(\cdot,t)) & = \sum\limits_{n=1}^{\infty}|(u_{0}, \psi_{n}(\beta;\cdot))|^2 E_\al(-\lambda_{n}(\beta)t^{\al})^2 \lambda_{n}(\beta)\\
 &\leq \sum\limits_{n=1}^{\infty}|(u_{0}, \psi_{n}(\beta;\cdot))|^2 E_\al(-\lambda_{n}(\beta)t^{\al})^2 \frac{\lambda_{n}(\beta)^2}{\lambda_1(\beta)}\leq \frac{C}{\lambda_1(\beta) t^{2\al}} ||u_{0}||^2_{L^2(\Omega)}
\end{align*}
from where we conclude that $u_\beta \in C((0,T];H^1(\Omega))$.
\par Now, for the series 
\begin{align*}
\Delta u_\be(\bx,t) = & \sum\limits_{n=1}^{\infty}(u_{0}, \psi_{n}(\beta;\cdot)) E_\al(-\lambda_{n}t^{\al}) (-\lambda_{n}(\beta))\psi_{n}(\be;\bx,t)
\end{align*}
it holds that 
\begin{align*}
||\Delta u_\be(\cdot,t)||^2_{L^2(\Omega)} = &   \sum\limits_{n=1}^{\infty} |(u_{0}, \psi_{n}(\beta;\cdot))|^2 |E_\al(-\lambda_{n}(\be)t^{\al})|^2 \lambda_{n}^2(\be) \\
& \le  \frac{C}{t^{2\al}} ||u_{0}||^2_{L^2(\Omega)},
 \end{align*}
which shows that the series
 is uniformly convergent in $t\in[\delta,T]$ for any given $\delta>0$ and thus, $ \Delta u_\be \in C((0,T];L^{2}(\Omega))$.  Taking into account that $\DC u_\be = \Delta u_\be$ in $L^2(\Omega)$ we also have that
\begin{equation*}
\DC u_\be \in C((0,T];L^2(\Omega)).
\end{equation*}


\par To finally see that we have constructed a weak solution, let us check that 
\begin{equation}\label{limite}
 \lim\limits_{t\rightarrow 0}\left\| u_\be(\cdot,t)-u_0 \right\|_{L^2(\Omega)}=0.
\end{equation}
Again we use the series expansions in $L^{2}(\Omega)$ and the properties of the Mittag-Leffler functions:
\begin{align*}
 \left\| u_\be(\cdot,t)-u_0 \right\|^2_{L^2(\Omega)} \le & \sum\limits_{n=1}^\infty \left|(u_{0},\psi_{n}(\be;\cdot))\right|^2 \left|E_\al(-\lambda_{n}(\beta)t^{\al})-1\right|^2 \\
\le & \sum\limits_{n=1}^\infty |(u_{0},\psi_{n}(\be;\cdot))|^2 \left[\left(\frac{C}{1+\lambda_{n}(\beta)t^{\al}} \right)^2+1\right], 
\end{align*}
which is convergent for every $t\in[0,T]$. Hence we can interchange the sum and the limit, and observing that $\lim\limits_{t\rightarrow 0} E_\al(-\lambda_{n}t^{\al})=1$, the desired limit \eqref{limite} holds.

\par The last step is to prove uniqueness. Let us consider $u_{0}\equiv 0$,  we will show that the only solution to the IBVP-$\be$ is the trivial one. Indeed, for each $n\in\mathbb{N}$ the problem \eqref{FODEIVP} has unique trivial solution $u_{n}(\beta,t)=0$ for $t\in [0,T]$. Hence the series that defines $u$ must be the zero function.

\endproof

\section{Convergences}

\par In this section we analyze the convergence of the solution $u_\be$ of the $\beta-$IBVP \eqref{beta-IBVP}, to the solution  $u_D$ of the $D-$IBVP \eqref{PD-beta}. 

\par Let us now study the very interesting relation between the eigenvalues and the eigenfunctions given in Theorem \ref{ExBaseH-PbST-D} and Theorem \ref{ExBaseH-PbST-beta}. In a recent work of Filinovsky, namely \cite{Fi:2014}, the following theorem is proved through variational techniques.
\begin{theorem}\label{teo-conv-autov} For $k\in\bbN$ and $\be>0$ the eigenvalues given in Theorems \ref{ExBaseH-PbST-D} and \ref{ExBaseH-PbST-beta}, enumerated
according to their multiplicities, satisfies the  following estimate 
\begin{equation}\label{desigLambda} 
0\le \la_k^D - \la_k(\be) \le C_{1}\be^{-\frac{1}{2}}(\la_k^D)^2,
\end{equation}
where the constant $C_{1}$ does not depend on $k$.
\end{theorem}
From \eqref{desigLambda} it is obvious that
\begin{equation}\label{lim-lambda(beta)}
\lim\limits_{\be\rightarrow \infty}\la_k(\be)=\lambda_k^D, \text{for every } \, k\in \bbN.
\end{equation}

The next Theorem deals with the weak convergence related to the eigenfunctions.

Let us observe that a similar result is given in an even more recent work \cite{Fi:2017} of Filinovsky, but we present a different proof according to our problems. 

\begin{theorem}\label{teo-weak-conv-autof} Let $\left\{\varphi_k\right\}_k$ and $\left\{\la_k^D\right\}_k$ be the sequence of eigenfunctions and eigenvalues of the  Dirichlet  problem  given in Theorem \ref{ExBaseH-PbST-D}, where $\la_1^D<\la_2^D\leq \ldots \rightarrow \infty$. And let $\left\{\psi_k(\be)\right\}_k$ and $\left\{\la_k(\be)\right\}_k$ be the sequence of eigenfunctions and eigenvalues of the  Robin eigenvalue problem given in  \eqref{ExBaseH-PbST-beta}, where $\la_1(\be)<\la_2(\be)\leq \ldots \rightarrow \infty$. Then for each $k=1,2,\ldots$ 
\begin{equation}\label{conv-debil}
\psi_k(\be) \rightharpoonup \varphi_k \quad \text{weak in } H^1(\Omega),\quad \text{when } \be \rightarrow \infty.  
\end{equation}   

\end{theorem}

\proof For each fixed $k \in \bbN$ let $\la_k(\beta)$ and $\la_k^D$ be the eigenvalues given in Theorems \ref{ExBaseH-PbST-beta} and \ref{ExBaseH-PbST-D}, respectively. Being $\varphi_k$ a function that verifies $(i)$ and $(ii)$ of Theorem \ref{ExBaseH-PbST-D}, and  $\psi_k(\beta)$ a function that verifies $(ii)$ and $(iii)$ of Theorem \ref{ExBaseH-PbST-beta}, we can consider the bilinear forms given in \eqref{a(u,v)} and \eqref{a_beta(u,v)} and affirm that:\\
i) The eigenfunction $\varphi_k$ is the unique solution to the variational problem: \textsl{Find the function} $u \in H^1_0(\Omega)$ such that   
\begin{equation}\label{EV-0}
a\left(u,v\right)=\la_k^D(u,v) \quad \text{ for every } \, v\in H^1_0(\Omega). 
\end{equation}
ii) The eigenfunction  $\psi_k(\beta)$ is the unique solution to the problem:
\textsl{Find the function} $u \in H^1(\Omega)$ such that   
\begin{equation}\label{EV-1}
a_\beta\left(u,v\right)=\la_k(\beta)(u,v) \quad \text{ for every } \, v\in H^1(\Omega).
\end{equation}

\noindent Replacing $v=\psi_k(\be)-\varphi_k \in H^1(\Omega)$ in \eqref{EV-1}, subtracting $a(\varphi_k,\psi_k(\be)-\varphi_k)$ to both members  we get
\begin{align*}a(\psi_k(\be)-\varphi_k,\psi_k(\be)-\varphi_k)+ \int\limits_{\partial\Omega} \left(\psi_k(\be)-\varphi_k\right)^2 \dd\gamma + (\be-1) \int\limits_{\partial\Omega} \left(\psi_k(\be)-\varphi_k\right)^2 \dd\gamma 
\end{align*}
\begin{align}\label{EV-2}  
\qquad \qquad \qquad  \qquad = \la_k(\beta)(\psi_k(\beta),\psi_k(\be)-\varphi_k)
-a(\varphi_k,\psi_k(\be)-\varphi_k).
\end{align} 

From the continuity of the bilinear form  $a$, the inequality \eqref{LemaTop},  being $\varphi_k=0$ in $\partial \Omega$, splitting $\be=(\be-1)+1$ for $\be>1$ as in \cite{Ta:1979},  and applying Theorem \ref{teo-conv-autov} we obtain
   
$$\hspace{-5cm} \eta_1||\psi_k(\be)-\varphi_k||_{H^1(\Omega)}^2 + (\be-1) \int\limits_{\partial\Omega} \left(\psi_k(\be)\right)^2 \dd\gamma \leq $$	
\begin{equation}\label{EV-3}
\hspace{2cm}\leq M_k ||\psi_k(\beta) ||_{L^2(\Omega)} ||\psi_k(\be)-\varphi_k ||_{L^2(\Omega)} + C||\varphi_k||_{H^1(\Omega)}||\psi_k(\be)-\varphi_k||_{H^1(\Omega)}. 
\end{equation}	

Naming $C_k:=M_k   +C||\varphi_k||_{H^1(\Omega)}$ we get 
\begin{equation}\label{EV-3}
\eta_1||\psi_k(\be)-\varphi_k||_{H^1(\Omega)}^2 + (\be-1) \int\limits_{\partial\Omega} \psi_k(\be)^2 \dd\gamma \leq C_k
||\psi_k(\be)-\varphi_k||_{H^1(\Omega)}. 
\end{equation}	
From \eqref{EV-3} we can state that 
\begin{equation}\label{EV-4}
||\psi_k(\be)-\varphi_k||_{H^1(\Omega)}\leq \frac{C_k}{\eta_1},
\end{equation}	
and from \eqref{EV-3} and  \eqref{EV-4}
\begin{equation}\label{EV-5}
(\be-1) \int\limits_{\partial\Omega} \psi_k(\be)^2 \dd\gamma \leq   \frac{C_k^2}{\eta_1}.
\end{equation}	
Inequality \eqref{EV-4} implies that $\left\{ \psi_k(\be) \right\}_\be$ is bounded in $H^1(\Omega)$, then there exists $\xi_k \in H^1(\Omega)$ such that 
\begin{equation}\label{EV-6}
\psi_k(\be) \rightharpoonup \xi_k \quad \text{in } H^1(\Omega) \text{weak },\quad \text{when } \be \rightarrow \infty.  
\end{equation}	
By taking the limit when $\be \rightarrow \infty$ in \eqref{EV-5} and  by using the lower semicontinuity of $v\rightarrow \int\limits_{\partial\Omega} v^2 \dd\gamma $, it holds that $\left. \xi_k \right|_{\partial\Omega}=0.$  Then 
\begin{equation}\label{EV-7}
 \xi \, \in \, H^1_0(\Omega).  
\end{equation}	
Finally, note that from \eqref{EV-1} we have that 
\begin{equation}\label{EV-8}
a\left(\psi_k(\beta),w\right)=\la_k(\beta)(\psi_k(\beta),w) \quad \text{ for every } \, w\in H^1_0(\Omega),
\end{equation}
Taking the limit when $\beta \rightarrow \infty$ in \eqref{EV-8} and using \eqref{EV-7} and Theorem \ref{teo-conv-autov} we obtain
\begin{equation}\label{EV-8'}
a\left(\xi_k,w\right)=\la_k^D(\xi_k,w) \quad \text{ for every } \, w\in H^1_0(\Omega),
\end{equation}
 Then, from \eqref{EV-8'} and the uniqueness of \eqref{EV-0} we conclude that $\xi_k=\varphi_k$ and the thesis holds. 
\endproof

\bigskip
\begin{theorem} The family of solutions $\left\{u_\be\right\}$ of the $\be-IBVP$ \eqref{beta-IBVP}, converges to the solution $u_D$ to the $D-IBVP$  in $L^2(\Omega)$  when $\be \to \infty$, for every $t\in (0,T)$. 
\end{theorem}

\begin{proof} 
Fix $t>0$. From \eqref{cota-uD-L2}  and \eqref{cota-u-beta-L2}  we can choose a  natural number $N$  such that 
\begin{equation}\label{uD-N1}
 \sum\limits_{k=N+1}^{\infty} |(u_{0}, \varphi_{k})|^2 |E_\al(-\lambda^D_{k}t^{\al})|^2 < \varepsilon^2
\end{equation}
and
\begin{equation}\label{ube-N2}
 \sum\limits_{k=N+1}^{\infty} |(u_{0}, \psi_{k}(\be))|^2 |E_\al(-\lambda_{k}(\be)t^{\al})|^2 < \varepsilon^2.
\end{equation}

Therefore, we have:
\begin{align} \nonumber
\left\|u_D(\cdot,t)-u_\be(\cdot, t)\right\|_{L^2(\Omega)} \leq & \left( \left\|A_N\right\|_{L^2(\Omega)} + \left\|B_N\right\|_{L^2(\Omega)} + \left\|C_N\right\|_{L^2(\Omega)}\right) \\ \label{uD-ube-1}
\leq &  \left\|A_N\right\|_{L^2(\Omega)} +2\varepsilon,
\end{align}
where 
\begin{align}\label{A_N}
A_N:= & \sum\limits_{k=1}^{N} \left[(u_{0}, \var_{k})E_\al(-\lambda^D_{k}t^{\al})\varphi_k - (u_{0}, \psi_{k}(\be))E_\al(-\lambda_{k}(\be)t^{\al})\psi_k(\beta)\right], \\ \label{B_N}
B_N:= &  \sum\limits_{k=N+1}^{\infty} (u_{0}, \var_{k})E_\al(-\lambda^D_{k}t^{\al})\varphi_k, \\ \label{C_N}
C_N:= & \sum\limits_{k=N+1}^{\infty} (u_{0}, \psi_{k}(\be))E_\al(-\lambda_{k}(\be)t^{\al})\psi_{k}(\be),
\end{align}
and we have applied inequalities \eqref{uD-N1} and \eqref{ube-N2}. Now

\begin{align} \nonumber
\left\|A_N\right\|_{L^2(\Omega)}   \leq & \sum\limits_{k=1}^{N} \left[\left\|(u_{0}, \var_{k})E_\al(-\lambda^D_{k}t^{\al})\varphi_k - (u_{0}, \psi_{k}(\be))E_\al(-\lambda^D_{k}t^{\al})\var_k\right\|_{L^2(\Omega)}  \right. \\ \nonumber
 & + \left\|(u_{0},\psi_{k}(\be))E_\al(-\lambda^D_{k}t^{\al})\varphi_k - (u_{0}, \psi_{k}(\be))E_\al(-\lambda_{k}(\be)t^{\al})\varphi_k \right\|_{L^2(\Omega)}   \\ \label{A_N-1}
& + \left.  \left\|(u_{0}, \psi_{k}(\be))E_\al(-\lambda_{k}(\be)t^{\al})\varphi_k - (u_{0}, \psi_{k}(\be))E_\al(-\lambda_{k}(\be)t^{\al})\psi_k(\beta) \right\|_{L^2(\Omega)}  \right].
\end{align}

Now, for each $k \in \left\{1,\dots,N\right\}$ we have:  
\begin{itemize}
\item[(I)] From Theorem \ref{teo-weak-conv-autof}, there exists $\be_{I,N}^k>0$ such that if $\be>\be_{I,N}^k$, then 
$$|(u_{0}, \var_{k})-(u_{0}, \psi_{k}(\be))|<\frac{\varepsilon}{N}.$$
\item[(II)] Using again Theorem \ref{teo-weak-conv-autof} we have that the real values $|(u_{0}, \psi_{k}(\be))|$ are bounded for $\beta$ big enoguh, that is,  there exist $\be_{II,N}^k>0$ and $M_k>0$ such that if $\be>\be_{II,N}^k$, then
$$|(u_{0}, \psi_{k}(\be))|<M_k.$$
\item[(III)] From the continuity of the Mittag-Leffler functions and Theorem \ref{teo-conv-autov}, there exists $\be_{III,N}^k>0$ such that if $\be>\be_{III,N}^k$, then $$|E_\al(-\lambda^D_{k}t^{\al}) - E_\al(-\lambda_{k}(\be)t^{\al})|<\frac{\varepsilon}{N}.$$
\item[(IV)] From Theorem \ref{teo-weak-conv-autof}, since weak convergence in $H^1$ gives strong convergence in $L^2$, there exists $\be_{IV,N}^k>0$ such that if $\be>\be_{IV,N}^k$, then $$\left\| \var_k -\psi_k(\beta) \right\|_{L^2(\Omega)}<\frac{\varepsilon}{N}.$$
\end{itemize}

Finally, taking $M= \max_{1\leq k \leq N}M_k$ and $\bar{\be}=\max\left\{ \be_{l,N}^k; l=I,II,III,IV; \, k=1, \dots,N\right\}$  we deduce that for $\beta>\bar{\be}$,
\begin{equation}\label{A_N-3}
\left\|A_N\right\|_2<(2M+1)\varepsilon.
\end{equation}
From \eqref{uD-ube-1} and  \eqref{A_N-3} we conclude that, for every $t\in (0,T)$ it holds that  \linebreak $\lim\limits_{\beta \rightarrow \infty} \left\|u_D(\cdot,t)-u_\be(\cdot, t)\right\|_{L^2(\Omega)}=0$, as desired. 
\end{proof}

\section{The one-dimensional case}

\par In order to illustrate the convergence result by the aid of some software, we set ourselves in the following one-dimensional setting: let the domain be the real unit interval $\Omega=[0,1]$, and let $T>0$ to be fixed, say $T=1$. Here, we have that the boundary of the domain consists of two points, namely $\partial\Omega=\{0,1\}$. For simplicity we will write $x\in\mathbb{R}$ instead of $\bx$,  $\Delta u(\bx,t) = u_{xx}(x,t)$, etc. For fixed $0<\alpha<1$ and each $\beta>0$, the initial-boundary value problems to be considered are the following
\begin{itemize}
\item The one dimensional Dirichlet problem, which we call D-IVBP-1d:
\begin{equation}\label{D-IBVP-1d}
\begin{array}{clll}
     (i)  &  _0^CD^\alpha_t u(x,t)= u_{xx}(x,t), &   x \in [0,1], \,  0<t<T,   \,\\
     (ii)  &  u(x,0)=u_{0}(x), &   x \in [0,1], \, \,\\
     (iii) &   u(0,t)=u(1,t)=0, &     0<t<T.                      \end{array}
                                             \end{equation}
\item The one dimensional Robin problem, which we call $\beta-$IBVP-1d
\begin{equation}{\label{beta-IBVP-1d}}
\begin{array}{clll}
     (i)  &  _0^CD^\alpha_t u(x,t)= u_{xx}(x,t), &   x \in [0,1], \,  0<t<T,  \,\\
     (ii)  &  u(x,0)=u_{0}(x), &   x \in [0,1], \, \,\\
     (iii) &  -u_{x}(0,t)+\beta u(0,t)=0, &     0<t<T, \, \,\\
           &  u_{x}(1,t)+\beta u(1,t)=0, &     0<t<T. \, \,\\
        \end{array}
\end{equation}
\end{itemize}

\par Let us construct the solutions to both problems by proceding like in Section 3.

\par For the D-IVBP-1d problem we have the next Sturm-Liouville type problem with Dirichlet condition:
\begin{equation}\label{D-SL}
D-SL\left\{ \begin{array}{cl}
        \psi''(x) + \lambda \psi(x) = 0, & 0<x<1, \\
        \psi(0)=\psi(1)=0. & 
\end{array} \right.
\end{equation}
The solutions to this problems are, for $k\in\mathbb{N}$, the pairs of eigenvalues and eigenfunctions $(\lambda_k,\psi_k)$ given by $\lambda_{k}=k^{2}\pi^{2}$ and $\psi_{k}(x)=\sqrt{2}\sin(k\pi x)$.
The solution to the ordinary FDE linked to the time variable  is given by the Mittag-Leffler function: $\eta(t)=E_\al(-\lambda t^{\alpha})$. Finally, we couple this solutions to get that the formal solution to the D-IVBP-1d is given by
$$u(x,t)=\sum\limits_{k=1}^{\infty} \left( u_{0},\sqrt{2}\sin(k\pi\cdot)\right) \sqrt{2}\sin(k\pi x)E_\al(-k^2\pi^2 t^\alpha).$$

\par For problem $\beta$-IVBP-1d, the Sturm-Liouville type problem with Robin condition is
\begin{equation}\label{R-SL}
R-SL\left\{ \begin{array}{ll}
        \psi''(x) + \lambda \psi(x) = 0,& 0<x<1  \\
        -\psi'(0)+\beta\psi(0)=0,& \\
				\psi'(1)+\beta\psi(1)=0.& \\
\end{array} \right.
\end{equation}
Working  with the general solutions to the second order differential equation above
 $$\psi(x)=A\sin(\sqrt{\lambda}x)+B\cos(\sqrt{\lambda}x),$$ and with the Robin conditions we get the following system:
\begin{equation*}
\left\{ \begin{array}{l}
        -A\sqrt{\lambda}+\beta B =0, \\
				(A\sqrt{\lambda}+\beta B)\cos(\sqrt{\lambda})+(\beta A - B \sqrt{\lambda})\sin(\sqrt{\lambda})=0,
\end{array} \right.
\end{equation*}
or equivalently
\begin{equation*}                                                                                                                            
\left\{ \begin{array}{l}
        B = A \frac{\sqrt{\lambda}}{\beta}, \\ \label{func.h.lam.k.be.1d}
				2\sqrt{\lambda}\cos(\sqrt{\lambda})+\left(\beta - \frac{\lambda}{\beta}\right)\sin(\sqrt{\lambda})=0.
\end{array} \right.
\end{equation*}
Note that from the last equation we get an implicit formula for $\lambda$, namely $\tan(\sqrt{\lambda})=\frac{2\sqrt{\lambda}\beta}{\lambda-\beta^2}$, from where it can be  observed that the equation becomes $\sin(\sqrt{\lambda})=0$, when $\beta$ approaches infinity. This is precisely the equation that defines the eigenvalues for the Dirichlet case.

Now, let us consider the functions 
\begin{equation}\label{func.h.lam.k.be.1d}
h_{\beta}(\lambda):=2\sqrt{\lambda}\cos(\sqrt{\lambda})+\left(\beta - \frac{\lambda}{\beta}\right)\sin(\sqrt{\lambda}).
\end{equation}
We know from Theorems \ref{ExBaseH-PbST-beta} and \ref{teo-conv-autov} that for fixed $\beta$, the values of $\lambda$ that satisfy $h_\beta(\lambda)=0$ are countable, say $\{\lambda_{k}(\beta)\}$ and moreover, they converge to $\{k^2\pi^2\}$ when $\beta$ goes to infinity. Hence we can isolate the roots of the functions $h_{\beta}$ around $k^2\pi^2$ and approximate them numerically using some software. We used SageMath to compile the following Table:
\begin{table}[ht]
\centering
\resizebox{\textwidth}{!}{\begin{tabular}{c|ccccc||c}
			$\lambda_{k}(\beta)$ & $\beta = 10^2$ & $\beta = 10^3$ & $\beta = 10^4$ & $\beta = 10^5$ & $\beta = 10^6$ & $\lambda_{k}=k^2\pi^2$\\ \hline
			k=1 & 9.486473204354914 &  9.830244232285152 &  9.865657743495532 &  9.869209628756735 &  9.869564922790271 &  9.869604401089359\\
			k=2 & 37.947300586356484 &  39.320978472148354 &  39.462630975538794 &  39.47683851502818 &  39.47825969116076 &  39.47841760435743\\
			k=3 & 85.38668247637567 &  88.47220734834096 &  88.79091970080098 &  88.82288665881923 &  88.8260843051117 &  88.82643960980423\\
			k=4 & 151.81154366014752 &  157.28393857454202 &  157.85052392706672 &  157.90735406013744 &  157.91303876464283 &  157.9136704174297\\
			k=5 & 237.23142256188495 &  245.75618294799932 &  246.64144366523558 &  246.73024071899425 &  246.73912306975478 &  246.7401100272340\\
			k=6 & 341.6583215827787 &  353.888954347627 &  355.1636789293196 &  355.2915466354033 &  355.30433722044694 &  355.3057584392169\\
			k=7 & 465.1065236769831 &  481.6822697315615 &  483.41722973644585 &  483.5912718093815 &  483.6086812167196 &  483.6106156533786\\ 
			k=8 & 607.5923811691522 &  629.1361491341754 &  631.4020961068547 &  631.6294162409495 &  631.6521550585727 &  631.6546816697189\\
			k=9 & 769.1340834087115 &  796.2506156625528 &  799.118278063901 &  799.4059799301307 &  799.4347587460061 &  799.4379564882380\\
			k=10 & 949.7514101063597 &  983.0256954924237 &  986.5657756340526 &  986.920962876951 &  986.9564922790203 &  986.9604401089359
	\end{tabular}}\caption{\footnotesize{Values for $\lambda_{k}(\beta)$ found as roots of function \eqref{func.h.lam.k.be.1d}}.}
\end{table} 

\newpage
\noindent Also by the aid of SageMath we were able to visualize the functions $h_{\beta}$ and its roots in Figure \ref{fig:graf1}. A zoomed version around $\pi^2$ can also be seen in Figure \ref{fig:graf2}.
\begin{figure}[h!]
  \centering
  \includegraphics[width=0.7\linewidth]{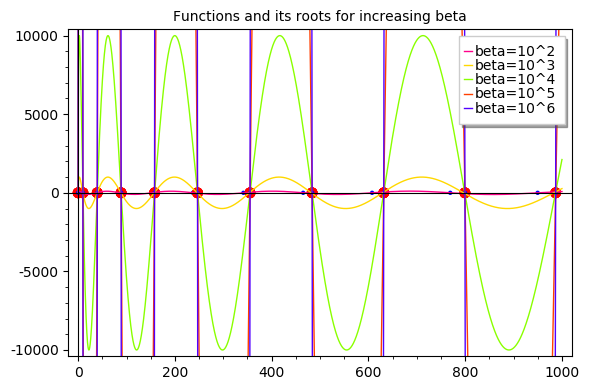}
  \caption{Functions \eqref{func.h.lam.k.be.1d} and roots for different values of $\beta$.}
  \label{fig:graf1}
\end{figure}
\begin{figure}[h!]
  \centering
  \includegraphics[width=0.7\linewidth]{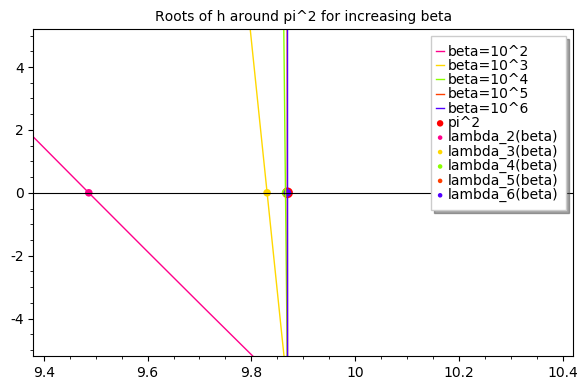}
  \caption{Close up of functions 
	for different values of $\beta$ around $\pi^2$.}
  \label{fig:graf2}
\end{figure}


For the entries in the above table, we have that $$\psi_{k}(\beta;x)=A_{k}(\beta)\left(\sin(\sqrt{\lambda_{k}(\beta)}x)+ \frac{\sqrt{\lambda_{k}(\beta)}}{\beta} \cos(\sqrt{\lambda_{k}(\beta)}x)\right),$$ where the coefficients $A_{k}(\beta)$ are such that the $\psi_{k}(\beta;\cdot)$'s are orthonormal in $L^2([0,1])$. By performing some straightforward calculations keeping in mind that \eqref{func.h.lam.k.be.1d} holds, we get that $$A_{k}(\beta)=\frac{\sqrt{2}\beta}{\sqrt{\beta^2+\beta+\lambda_{k}(\beta)}},$$ which, as expected, converges to $\sqrt{2}$ when $\beta$ goes to infinity.

Last but not least, for finding the appropiate coefficients $\left( u_{0},\psi_{k}(\beta;\cdot)\right)$ we resort to the initial value. Again, coupling these solutions provides us with the formal solutions
\begin{align*}
u_{\beta}(x,t)=& \sum\limits_{k=1}^{\infty} \left( u_{0},\psi_{k}(\beta;\cdot)\right) \times \\
& \times A_{k}(\beta) \left(\sin\sqrt{\lambda_{k}(\beta)}x+ \frac{\sqrt{\lambda_{k}(\beta)}}{\beta} \cos\sqrt{\lambda_{k}(\beta)}x \right) E_\al(-\lambda_{k}(\beta) t^\alpha),
\end{align*}
and, if we observe the explicit formulation for $u_{\beta}(x,t)$ when $\beta$ goes to infinity, it is clear that we obtain $u(x,t)$.

If we fix some more values and functions we will be able to have an idea of what these solutions look like. For the software implementation we will be considering $\alpha=0.8$, $u_{0}=\sin(\pi x)$ and $\beta=10^{l}$ for $l=1,2,3,4,5$. Of course, all of these values could be changed at will. We chose this initial data in order to have the solution to the D-IVBP-1d consisting of only one term, and the values of $\beta$ bigger than $10^4$ provided us with no visible changes on the outcome. The solution to the  D-IVBP-1d problem  is then $$u(x,t)= \sin(\pi x)E_\al(-\pi^2 t^\alpha).$$ Let us remark that in order to plot this function for fixed time $t=1$, we had to use the integral expression for the Mittag-Leffler functions given in Theorem 2.1  from \cite{GoLuLu:2002}, since from the definition only it seemed that the software brought precision loss due to numeric issues. Also, as it is expected, only the first term of the Fourier series defining $u_{\beta}$ has a visual impact, from the second one on it really doesn't affect the visualization. Figure \ref{fig:graf3} shows how the solutions of the $\beta$-IVBP-1d approaches the solution of the D-IVBP-1d.
\begin{figure}[h!]
  \centering
  \includegraphics[width=0.7\linewidth]{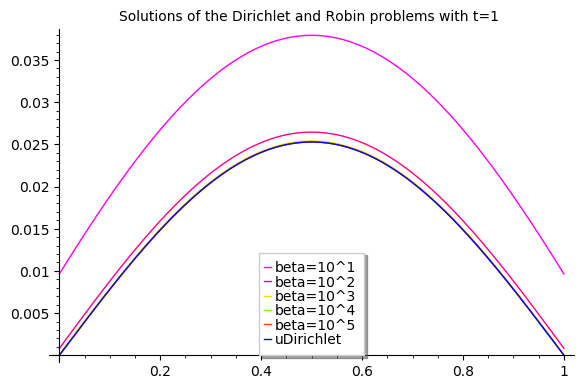}
  \caption{Solutions for the $\beta$-IVBP-1d problems vs. the solution for the D-IVBP-1d problem at time $t=1$.}
  \label{fig:graf3}
\end{figure}

\section{Conclusion}
We have proved existence and uniqueness of solutions to a family of  initial boundary value problems with Robin condition  for the FDE where  the parameter of the family was  the Newton heat transfer coefficient linked to the Robin condition on the boundary.The proofs where done following the Fourier approach and the convergence result lead us to a limit problem which is  the  initial boundary value problem for the FDE with an homogeneous Dirichlet condition. Finally we have visualized the previous results by considering a one-dimensional case with the aid of SageMath software.  \\

\section{Acknowledgements}
\noindent The authors thank Prof. Mashiro Yamamoto for his kindly and fruitful dicussion. 
\noindent The present work has been sponsored by the Projects ANPCyT PICTO Austral 2016 N$^\circ 0090$, Austral N$^\circ$ 006-INV00020 (Rosario, Argentina), European Unions Horizon 2020 research and innovation  programme under the Marie Sklodowska-Curie Grant Agreement N$^\circ$ 823731 CONMECH and SECyT-Univ. Nac. de Rosario.

\bibliographystyle{plain}

\bibliography{Roscani_BIBLIO_GENERAL_nombres_largos_2021}

\end{document}